\documentclass[reqno,12pt,letterpaper]{amsart}
\usepackage{amsmath,amssymb,amsthm,graphicx,mathrsfs,url}
\usepackage[usenames,dvipsnames]{color}
\usepackage[colorlinks=true,linkcolor=Red,citecolor=Green]{hyperref}
\usepackage{amsxtra}
\usepackage{wasysym} 
\usepackage{graphicx}
\usepackage{subcaption}

\usepackage{graphicx,color}

\def\arXiv#1{\href{http://arxiv.org/abs/#1}{arXiv:#1}}
\usepackage{array}
\newcolumntype{P}[1]{>{\centering\arraybackslash}m{#1}}

\def\wrtext#1{\relax\ifmmode{\leavevmode\hbox{#1}}\else{#1}\fi}

\def\abs#1{\left|#1\right|}

\def\?[#1]{\textbf{[#1]}\marginpar{\Large{\textbf{??}}}}

\let\epsilon=\varepsilon 

\def\norm#1{||\,#1\,||}

\newtheorem{theo}{Theorem}
\newtheorem{prop}{Proposition}[section]

\newtheorem{lemm}[prop]{Lemma}

\numberwithin{equation}{section}

\usepackage{scalerel}

\newcommand\reallywidehat[1]{\arraycolsep=0pt\relax%
\begin{array}{c}
\stretchto{
  \scaleto{
    \scalerel*[\widthof{\ensuremath{#1}}]{\kern-.5pt\bigwedge\kern-.5pt}
    {\rule[-\textheight/2]{1ex}{\textheight}} 
  }{\textheight} %
}{0.5ex}\\           
#1\\                 
\rule{-1ex}{0ex}
\end{array}
}

\begin{document}

\title[Pullback operators on Bargmann spaces]{Pullback operators on Bargmann spaces}

\author{Reid Johnson}
\address{Department of Mathematics, University of California, Los Angeles, CA 90095, USA.}
\email{reid@math.ucla.edu}

\begin{abstract}
We characterize boundedness and compactness of pullback operators under holomorphic maps between Bargmann spaces of entire holomorphic functions with quadratic strictly plurisubharmonic exponential weights, extending the result of~\cite{model}, obtained in the case of the radial quadratic weight. We also show that the pullback operator between Bargmann spaces is compact precisely when it is of trace class, with sub-exponentially decaying singular values.
\end{abstract}

\maketitle

\section{Introduction and statement of results}
\medskip
\noindent
The study of pullback (composition) operators on spaces of holomorphic functions has a long tradition in complex analysis and operator theory~\cite{CoMc},~\cite{GoHe}. In the setting of the radial Bargmann space
\begin{equation}
\label{eq1.0.1}
H_{\Phi_0}(\mathbb C^n) = {\rm Hol}(\mathbb C^n)\cap L^2(\mathbb C^n, e^{-2\Phi_0}\, L(dx)), \quad \Phi_0(x) = \frac{1}{4}\abs{x}^2,
\end{equation}
where $L(dx)$ stands for the Lebesgue measure on $\mathbb{C}^{n}$, such operators were considered in~\cite{model}, and a complete characterization of bounded and compact pullback operators on $H_{\Phi_0}(\mathbb C^n)$ has been established in~\cite{model}. Now in numerous situations one naturally encounters Bargmann spaces of form (\ref{eq1.0.1}), where the radial weight $\Phi_0$ is replaced by a more general strictly plurisubharmonic quadratic form $\Phi$ on $\mathbb C^n$,
\begin{equation}
\label{eq1.1}
\sum_{k,\ell=1}^n \frac{\partial^2 \Phi(x)}{\partial x_k \partial \overline{x}_{\ell}} \zeta_k \overline{\zeta}_{\ell} > 0, \quad x\in \mathbb C^n,\,\,\,0 \neq \zeta \in \mathbb C^n.
\end{equation}
Here and in what follows we use the standard notation
\[
\partial_{x_j} = \frac{1}{2}\left(\partial_{{\rm Re}\, x_j} - i \partial_{{\rm Im}\, x_j}\right), \quad \partial_{\overline{x}_j} = \frac{1}{2}\left(\partial_{{\rm Re}\, x_j} + i \partial_{{\rm Im}\, x_j}\right),\quad 1\leq j \leq n.
\]
In the theory of metaplectic FBI transforms~\cite{Sj95},~\cite{HiSj15}, Bargmann spaces
\begin{equation}
\label{eq1.1.1}
H_{\Phi}(\mathbb C^n) = {\rm Hol}(\mathbb C^n)\cap L^2(\mathbb C^n, e^{-2\Phi}\, L(dx)),
\end{equation}
where $\Phi$ is a strictly plurisubharmonic quadratic form on $\mathbb{C}^n$, occur naturally as the image spaces of such transforms acting on $L^2(\mathbb R^n)$, and the flexibility afforded by this approach has been very useful in several works on spectral theory of non-self-adjoint operators,~\cite{Hi04},~\cite{HiSjVi13},~\cite{CGHS},~\cite{HiPrVi},~\cite{Vi22}, in particular. It seems therefore natural to try to extend the results of~\cite{model} to pullback operators under holomorphic maps acting on Bargmann spaces of the form (\ref{eq1.1.1}), with $\Phi$ quadratic strictly plurisubharmonic, and it is precisely the purpose of this note to do so. Along the way, we shall be able to sharpen the description of pullback operators that are compact, and it also seems that our proofs may be of some independent interest even when specialized to the radial case (\ref{eq1.0.1}). Let us now proceed to describe the assumptions and state the main results of this work.

\medskip
\noindent
Let $\Phi_1$, $\Phi_2$ be strictly plurisubharmonic quadratic forms on $\mathbb{C}^{n_1}$, $\mathbb{C}^{n_2}$, respectively. Associated to the quadratic forms $\Phi_j$ are the corresponding Bargmann spaces $H_{\Phi_j}(\mathbb C^{n_j})$, $j=1,2$, defined as in (\ref{eq1.1.1}).
Let $\varphi:\mathbb{C}^{n_2}\rightarrow\mathbb{C}^{n_1}$ be a holomorphic map and let us define the corresponding pullback (composition) operator
\begin{equation}
\label{eq1.2}
C_\varphi u :=u\circ\varphi,
\end{equation}
for $u:\mathbb{C}^{n_1}\rightarrow\mathbb{C}$ holomorphic. It will be fruitful to view $C_{\varphi}$ as a metaplectic Fourier integral operator in the complex domain, see~\cite[Appendix B]{CGHS} and the discussion in Section \ref{suff_cond} below. Our first result, which is an extension of~\cite[Theorem 1, Theorem 2]{model}, is as follows.

\begin{theo}
\label{theo 1}
The pullback operator
\[
C_\varphi: H_{\Phi_1}(\mathbb{C}^{n_1}) \rightarrow H_{\Phi_2}(\mathbb{C}^{n_2})
\]
is bounded if and only if the following conditions are satisfied:
\begin{enumerate}
\item \label{affine}There exist a complex $n_1\times n_2$ matrix $A$ and $b\in\mathbb{C}^{n_1}$ such that
\begin{equation}
\label{eq1.3}
\varphi(x)=Ax+b, \quad x\in \mathbb C^{n_2}.
\end{equation}
\item \label{kernel}The matrix $A$ satisfies \begin{equation}
\label{kernel condition}
    \Phi_2(x) \leq 0 \Longrightarrow Ax \neq 0, \quad 0 \neq x\in\mathbb{C}^{n_2}.
\end{equation}
\item \label{bounded} We have
\begin{equation}
\label{eq1.4}
e^{\Phi_1\circ\varphi-\Phi_2} \in L^{\infty}(\mathbb C^{n_2}).
\end{equation}
\end{enumerate}
\end{theo}

\medskip
\noindent
{\it Remark}. Condition (\ref{eq1.4}), where $\varphi$ is given by (\ref{eq1.3}), holds precisely when the real quadratic polynomial $\mathbb C^{n_2} \ni x \mapsto \Phi_1(Ax + b) - \Phi_2(x)$ is bounded above on $\mathbb C^{n_2}$. Using Taylor's formula we may write
\begin{equation}
\label{eq1.5}
\Phi_1(Ax + b) = \Phi_1(Ax) + 2{\rm Re}\, \left((\partial_x \Phi_1)(Ax) \cdot b\right) + \Phi_1(b),
\end{equation}
and therefore (\ref{eq1.4}) holds precisely when
\begin{equation}
\label{eq1.6}
\Phi_1(Ax) - \Phi_2(x) \leq 0, \quad x\in \mathbb C^{n_2},
\end{equation}
and
\begin{equation}
\label{eq1.7}
\Phi_1(Ax) - \Phi_2(x) = 0 \Longrightarrow {\rm Re}\, \left((\partial_x \Phi_1)(Ax) \cdot b\right) = 0.
\end{equation}
It is in this form that the boundedness result is stated in~\cite{model}, when $n_1 = n_2$ and $\displaystyle \Phi_1(x) = \Phi_2(x) = \Phi_0(x)$ in (\ref{eq1.0.1}). Condition (\ref{kernel}) holds automatically in the latter case.

\medskip
\noindent
The compactness of the pullback operator (\ref{eq1.2}) as a map: $H_{\Phi_1}(\mathbb{C}^{n_1}) \rightarrow H_{\Phi_2}(\mathbb{C}^{n_2})$
can also be characterized. The following result provides an extension of~\cite[Theorem 1, Theorem 2]{model} and furnishes some additional quantitative information.

\bigskip
\noindent
\begin{theo}
\label{theo 2}
The following conditions are equivalent:
\begin{enumerate}
\item\label{trace} The pullback operator
$$
C_\varphi: H_{\Phi_1}(\mathbb{C}^{n_1}) \rightarrow H_{\Phi_2}(\mathbb{C}^{n_2})
$$
is of trace class, with the singular values $s_j(C_{\varphi})$ satisfying
\begin{equation}
\label{eq1.7.1}
s_j(C_{\varphi}) \leq C \exp\left(-\frac{j^{1/n_2}}{C}\right), \quad j =1,2,\ldots ,
\end{equation}
for some $C>0$.
\item\label{compact} The pullback operator
$$
C_\varphi: H_{\Phi_1}(\mathbb{C}^{n_1}) \rightarrow H_{\Phi_2}(\mathbb{C}^{n_2})
$$
is compact.
\item \label{neg def} There exist a complex $n_1\times n_2$ matrix $A$ and $b\in\mathbb{C}^{n_1}$ such that
$$
\varphi(x)=Ax+b, \quad x\in \mathbb C^{n_2},
$$
and the quadratic form $\mathbb{C}^{n_2}\ni x\mapsto \Phi_1(Ax)-\Phi_2(x)$ is negative definite,
\begin{equation}
\label{eq1.8}
\Phi_1(Ax)-\Phi_2(x) < 0,\quad 0\neq x \in \mathbb C^{n_2}.
\end{equation}
\end{enumerate}
\end{theo}

\medskip
\noindent
{\it Remark}. In the case of the radial weight $\Phi_0$ in (\ref{eq1.0.1}), it was remarked in \cite{JiPrZh}, \cite{Que17} that the pullback operator $C_{\varphi}: H_{\Phi_0}(\mathbb C^n) \rightarrow H_{\Phi_0}(\mathbb C^n)$ belongs to all Schatten-von Neumann classes $C_p(H_{\Phi_0}(\mathbb C^n))$, $p > 0$, as soon as it is compact, and the exponential decay of the singular values of $C_{\varphi}$ in this case has been established in~\cite{Que17} in dimension $n = 1$. To the best of our knowledge, the sub-exponential decay of the singular values (\ref{eq1.7.1}) in higher dimensions has not been shown previously, even for the radial weight.

\medskip
\noindent
{\it Remark}. The work \cite{Le17} characterized the boundedness and compactness of the pullback operator $C_\varphi:H_{\Phi_1}(\mathbb{C}^{n_1}) \rightarrow H_{\Phi_2}(\mathbb{C}^{n_2})$ for radial weights $\Phi_1(x) = \frac{1}{2}|x|^2$, $x\in\mathbb{C}^{n_1}$ and $\Phi_2(x) = \frac{1}{2}|x|^2$, $x\in\mathbb{C}^{n_2}$, essentially extending the results of \cite{model} to the case when $n_1,n_2$ are not necessarily equal. Hence, our characterization results can be seen as further extensions of those in \cite{Le17}.

\medskip
\noindent
{\it Remark}. Quite general plurisubharmonic weights on $\mathbb C^n$ are considered in the work \cite{ArTo}, and the boundedness and compactness of (weighted) composition operators on the corresponding exponentially weighted spaces of holomorphic functions are characterized in~\cite{ArTo},  in terms of certain integral transforms associated to the weight. Restricting the attention in Theorem \ref{theo 1} and Theorem \ref{theo 2} to strictly plurisubharmonic weights that are quadratic, we are able to obtain characterization results that are quite explicit.

\medskip
\noindent
{\it Example}. We shall briefly interpret our results in the case when $n_1=n_2=1$. For $j=1,2$, let $\Phi_j$ be a strictly plurisubharmonic quadratic form on $\mathbb{C}$, meaning there are constants $r_j>0$, $s_j\in\mathbb{C}$ such that
\begin{equation}
\label{1d spshqf}
    \Phi_j(x)=r_j\abs{x}^2 + {\rm Re}\,(s_jx^2), \quad x\in\mathbb{C}.
\end{equation}
For $\varphi:\mathbb{C}\rightarrow\mathbb{C}$ holomorphic, Theorem \ref{theo 1} implies that $C_\varphi:H_{\Phi_1}(\mathbb{C})\rightarrow H_{\Phi_2}(\mathbb{C})$ is bounded if and only if the following conditions are satisfied:
\begin{enumerate}
\item There exist complex numbers $a,b\in\mathbb{C}$ such that
\begin{equation}
\label{1d affine}
\varphi(x)=ax+b, \quad x\in \mathbb C.
\end{equation}
\item If $\abs{s_2}/r_2 \geq 1$ (i.e. $\Phi_2$ is not positive definite), then $a\neq 0$.
\item \label{1d bounded}
We have 
\begin{equation}
\label{1d bound}
    \abs{s_1a^2-s_2} \leq r_2 - r_1\abs{a^2},
\end{equation}
and $b$ belongs to the real subspace of $\mathbb{C}$ determined by
\begin{equation}
    \quad\quad\quad(s_1a^2-s_2) x^2 = r_2 - r_1\abs{a^2} \Longrightarrow {\rm Re}\, ((r_1\overline{ax}+s_1ax)b) = 0,\quad x\in S^1 \subset \mathbb{C}.
\end{equation}
\end{enumerate}
Here we have used the equivalence of (\ref{eq1.4}) with (\ref{eq1.6}), (\ref{eq1.7}) to simplify condition (\ref{1d bounded}). Theorem \ref{theo 2} implies that $C_\varphi:H_{\Phi_1}(\mathbb{C})\rightarrow H_{\Phi_2}(\mathbb{C})$ is compact if and only if $\varphi$ is of the form (\ref{1d affine}), and (\ref{1d bound}) holds with strict inequality. We observe in particular that bounded composition operators between fixed Bargmann spaces need not exist if the strictly plurisubharmonic weight associated with the target Bargmann space is "too far from" positive definite. Indeed, using our one dimensional characterizations above, one can show that for $\Phi_1$, $\Phi_2$ as in (\ref{1d spshqf}), there exists a bounded pullback operator from $H_{\Phi_1}(\mathbb{C})$ to $H_{\Phi_2}(\mathbb{C})$ if and only if $\abs{s_2}/r_2 < 1$ or $\abs{s_2}/r_2 \leq \abs{s_1}/r_1$, while there exists a compact pullback operator from $H_{\Phi_1}(\mathbb{C})$ to $H_{\Phi_2}(\mathbb{C})$ if and only if $\abs{s_2}/r_2 < 1$ or $\abs{s_2}/r_2 < \abs{s_1}/r_1$.

\medskip
\noindent
The plan of the note is as follows. The necessity part of Theorem \ref{theo 1} is proved in Section \ref{sec_nec}, proceeding essentially as in~\cite{model} and letting the (bounded) pullback operator, or rather its adjoint, act on the space of normalized coherent states for the Bargmann space. See also~\cite{chs}. The sufficiency part of Theorem \ref{theo 1} is then established in Section \ref{suff_cond}. Compared with ~\cite{model}, where a major idea in the proof of sufficiency is to use the singular value decomposition and to exploit the orthogonality of monomials in the radial Bargmann space $H_{\Phi_0}(\mathbb C^n)$, we proceed differently, viewing $C_{\varphi}$ as a complex Fourier integral operator and obtaining its boundedness by an application of Schur's lemma. Section \ref{sect_compact} is devoted to the proof of Theorem \ref{theo 2}. Our notation is standard and we follow the convention that $C$ denotes a positive constant whose value may change from line to line.

\medskip
\noindent
{\bf Acknowledgments}.
The author would like to thank Michael Hitrik for recommending this project, and for many essential suggestions. The author would also like to thank the anonymous referee for their careful reading and helpful comments.

\section{Necessary conditions for boundedness of pullback operators}
\label{sec_nec}
\medskip
\noindent
The purpose of this section is to prove the forward direction of Theorem~\ref{theo 1}. It will be helpful to decompose
\begin{equation}
\label{eq2.1}
\Phi_1 = \Phi_{{\rm herm}} + \Phi_{{\rm plh}},
\end{equation}
where $\Phi_{{\rm herm}}(x) = \left(\Phi_1\right)''_{\overline{x}x}x\cdot \overline{x}$ is Hermitian positive definite and $\Phi_{{\rm plh}}(x) = {\rm Re}\, f(x)$, with $f(x) = \left(\Phi_1\right)''_{xx}x\cdot x$, is pluriharmonic on $\mathbb C^{n_1}$.

\medskip
\noindent
We need to establish that three conclusions (\ref{affine}), (\ref{kernel}), (\ref{bounded}) in Theorem \ref{theo 1} follow from the boundedness of the operator
$$
C_\varphi: H_{\Phi_1}(\mathbb C^{n_1}) \rightarrow H_{\Phi_2}(\mathbb C^{n_2}),
$$
which we accomplish (out of order) in the following three lemmas:
\begin{lemm}
Let $\Phi_1$, $\Phi_2$ be strictly plurisubharmonic quadratic forms on $\mathbb{C}^{n_1}$, $\mathbb{C}^{n_2}$, respectively. If the pullback operator $C_\varphi: H_{\Phi_1}(\mathbb{C}^{n_1}) \rightarrow H_{\Phi_2}(\mathbb{C}^{n_2})$ is bounded for some holomorphic $\varphi:\mathbb C^{n_2} \rightarrow \mathbb C^{n_1}$, then we have
$$
e^{\Phi_1\circ\varphi-\Phi_2} \in L^{\infty}(\mathbb{C}^{n_2}).
$$
\end{lemm}
\begin{proof}

Following~\cite{model}, we shall consider the action of the adjoint $$C_{\varphi}^*:H_{\Phi_2}(\mathbb C^{n_2})\rightarrow H_{\Phi_1}(\mathbb C^{n_1})$$ on the space of "coherent states", i.e. the normalized reproducing kernels for the Bargmann space $H_{\Phi_2}(\mathbb C^{n_2})$. For $j=1,2$, let us set
\begin{equation}
\label{eq2.6}
k_{j,w}(x) = b_{\Phi_j}\, e^{2\Psi_j(x,\overline{w})-\Phi_j(w)},\quad w\in \mathbb C^{n_j}.
\end{equation}
Here $\Psi_j$ is the polarization of $\Phi_j$, i.e. the unique holomorphic quadratic form on $\mathbb C^{2n_j}$ such that $\Psi_j(x,\overline{x}) = \Phi_j(x)$, $x\in \mathbb C^{n_j}$, and $b_{\Phi_j} > 0$ is a suitable constant depending on $\Phi_j$ only, such that $\norm{k_{j,w}}_{H_{\Phi_j}(\mathbb C^{n_j})} = 1$, $w\in \mathbb C^{n_j}$. Let us recall from~\cite[Section 1]{Sj95},~\cite[Chapter 13]{Zw12} that the orthogonal projection
\begin{equation}
\label{eq2.2}
\Pi_{\Phi_j}: L^2(\mathbb C^{n_j}, e^{-2\Phi_j}L(dx)) \rightarrow H_{\Phi_j}(\mathbb C^{n_j}), \quad j =1,2
\end{equation}
is given by
\begin{equation}
\label{eq2.3}
\Pi_{\Phi_j} u(w) = b_{\Phi_j} e^{\Phi_j(w)} \langle u, k_{j,w}\rangle_{L^2(\mathbb{C}^{n_j},\, e^{-2\Phi_j} L(dx))} = b_{\Phi_j}^2 \int_{\mathbb{C}^{n_j}} e^{2\Psi_j(w,\overline{x})} u(x)\, e^{-2\Phi_j(x)}\, L(dx).
\end{equation}
Here we have also used that $\overline{\Psi_j(x,\overline{w})} = \Psi_j(w,\overline{x})$, $x,w\in \mathbb C^{n_j}$. By using (\ref{eq2.3}) to rewrite both sides of the trivial equation
\begin{equation}
    (\Pi_{\Phi_2}C_\varphi u)(w)=(C_\varphi u)(w)=u(\varphi(w))=(\Pi_{\Phi_1}u)(\varphi(w)), \quad u\in H_{\Phi_1}(\mathbb{C}^{n_1}),
\end{equation}
it is easy to see that 
\begin{equation}
    C_\varphi^*k_{2,w}=\frac{b_{\Phi_1}}{b_{\Phi_2}}e^{\Phi_1(\varphi(w))-\Phi_2(w)}k_{1,\varphi(w)},\quad w\in\mathbb{C}^{n_2},
\end{equation}
and the boundedness of $C_{\varphi}: H_{\Phi_1}(\mathbb C^{n_1}) \rightarrow H_{\Phi_2}(\mathbb C^{n_2})$ implies therefore that the norms
\begin{equation}
\label{eq2.10}
\norm{C_{\varphi}^* k_{2,w}}_{H_{\Phi_1}(\mathbb C^{n_1})} = \frac{b_{\Phi_1}}{b_{\Phi_2}}\, e^{\Phi_1(\varphi(w)) - \Phi_2(w)}, \quad w\in \mathbb C^{n_2},
\end{equation}
are bounded uniformly in $w\in \mathbb C^{n_2}$, completing the proof.
\end{proof}

\bigskip
\noindent
\begin{lemm}
\label{bounded implies affine}
Let $\Phi_1$, $\Phi_2$ be real quadratic forms on $\mathbb{C}^{n_1}$, $\mathbb{C}^{n_2}$, respectively, with $\Phi_1$ strictly plurisubharmonic. If \begin{equation}
\label{eq2.10.1}
e^{\Phi_1\circ\varphi-\Phi_2} \in L^{\infty}(\mathbb C^{n_2})
\end{equation}
for some holomorphic $\varphi:\mathbb C^{n_2} \rightarrow \mathbb C^{n_1}$, then we must have $\varphi(x)=Ax+b$, $x\in \mathbb C^{n_2}$, for some complex $n_1\times n_2$ matrix $A$ and some $b\in\mathbb C^{n_1}$.
\end{lemm}
\begin{proof}
Assumption (\ref{eq2.10.1}) holds precisely when the function $\Phi_1\circ\varphi-\Phi_2$ is bounded above on $\mathbb{C}^{n_2}$. We then get, recalling the decomposition (\ref{eq2.1}),
\begin{equation}
\label{eq2.11}
\frac{1}{C_1} \abs{\varphi(x)}^2 + {\rm Re}\, f(\varphi(x)) \leq C\left(1 + \abs{x}^2\right),\quad x\in \mathbb C^{n_2},
\end{equation}
for some constants $C>0$, $C_1 \geq 1$. Here we have also used that
\begin{equation}
\label{eq2.12}
\Phi_2(x) \leq C \abs{x}^2, \quad \frac{1}{C_1}\abs{y}^2 \leq \Phi_{\rm herm}(y),\quad x\in \mathbb C^{n_2},y\in\mathbb{C}^{n_1}.
\end{equation}
Letting $B(0,R)$ be the open ball of radius $R>0$ in $\mathbb C^{n_2}$ and writing $|\cdot|$ for the Lebesgue measure on $\mathbb C^{n_2}$, we get in view of (\ref{eq2.11}),
\begin{multline}
\label{eq2.13}
\frac{1}{C_1 \abs{B(0,R)}} \int_{B(0,R)} \abs{\varphi(x)}^2\, L(dx) + \frac{1}{\abs{B(0,R)}} \int_{B(0,R)} {\rm Re}\, f(\varphi(x))\, L(dx) \\
\leq C(1 + R^2),\quad R >0.
\end{multline}
See also~\cite[Proposition 2.1]{Le14} for a related idea. The mean value property for the entire holomorphic function $\mathbb C^{n_2} \ni x \mapsto f(\varphi(x))$ gives that
\begin{equation}
\label{eq2.14}
f(\varphi(0)) = \frac{1}{\abs{B(0,R)}} \int_{B(0,R)} f(\varphi(x))\, L(dx)
\end{equation}
and combining (\ref{eq2.13}), (\ref{eq2.14}), we obtain
\begin{equation}
\label{eq2.15}
\frac{1}{{\abs{B(0,R)}}}\int_{B(0,R)}\abs{\varphi(x)}^2\, L(dx) \leq C(1 + R^2),\quad R>0.
\end{equation}
Another application of the mean value property to the holomorphic function $\varphi$ together with the Cauchy-Schwarz inequality and (\ref{eq2.15}) implies that
\begin{multline}
\label{eq2.16}
\abs{\varphi(x)}^2 \leq \frac{1}{\abs{B(x,\abs{x})}}\int_{B(x,\abs{x})}\abs{\varphi(y)}^2\, L(dy) \\
\leq \frac{1}{\abs{B(x,\abs{x})}} \int_{B(0,2\abs{x})}\abs{\varphi(y)}^2\, L(dy) \leq C\left(1 + \abs{x}^2\right), \quad 0 \neq x \in\mathbb C^{n_2},
\end{multline}
and the result follows by Liouville's theorem.
\end{proof}

\medskip
\noindent
\begin{lemm}
Let $\Phi_1$, $\Phi_2$ be strictly plurisubharmonic quadratic forms on $\mathbb{C}^{n_1}$, $\mathbb{C}^{n_2}$, respectively, and let $\varphi(x)=Ax+b$, for some complex $n_1\times n_2$ matrix $A$ and $b\in\mathbb C^{n_1}$. If the operator $C_\varphi: H_{\Phi_1}(\mathbb C^{n_1}) \rightarrow H_{\Phi_2}(\mathbb C^{n_2})$ is bounded, then for any $0 \neq x\in\mathbb C^{n_2}$ with $\Phi_2(x) \leq 0$, we have $Ax \neq 0$.
\end{lemm}
\begin{proof}
The decomposition (\ref{eq2.1}) shows that $e^f \in H_{\Phi_1}(\mathbb C^{n_1})$ and therefore we have
$C_{\varphi}(e^f) = e^{f\circ \varphi} \in H_{\Phi_2}(\mathbb C^{n_2})$. In other words, the integral
\begin{equation}
\norm{C_\varphi(e^f)}^2_{H_{\Phi_2}(\mathbb{C}^{n_2})} = \int_{\mathbb{C}^{n_2}}\abs{e^{f(\varphi(x))}}^2 e^{-2\Phi_2(x)}\,L(dx)
    \\ = \int_{\mathbb{C}^{n_2}} e^{2({\rm Re}\, f(\varphi(x)))-2\Phi_2(x)}\, L(dx)
\end{equation}
converges, and the principal part of the real quadratic polynomial ${\rm Re}\, f(\varphi(x))-\Phi_2(x)$ on $\mathbb C^{n_2}$, given by the quadratic form ${\rm Re}\, f(Ax)-\Phi_2(x)$, is therefore negative definite,
$$
{\rm Re}\, f(Ax)-\Phi_2(x) < 0, \quad 0\neq x\in \mathbb C^{n_2}.
$$
The result follows.
\end{proof}

\section{Sufficient conditions for boundedness of pullback operators}
\label{suff_cond}
\medskip
\noindent
The purpose of this section is to prove the backward direction of Theorem \ref{theo 1}, so we can assume that conditions $(\ref{affine})$, $(\ref{kernel})$, and $(\ref{bounded})$ in Theorem \ref{theo 1} are all satisfied. Given $u\in H_{\Phi_1}(\mathbb C^{n_1})$, let us write using (\ref{eq2.2}), (\ref{eq2.3}),
\begin{equation}
\label{eq3.1}
(C_{\varphi} u)(x) = u(\varphi(x)) = (\Pi_{\Phi_1} u)(\varphi(x)) = b_{\Phi_1}^2 \int_{\mathbb C^{n_1}} e^{2\Psi_1(\varphi(x),\overline{y})} u(y)\, e^{-2\Phi_1(y)}\, L(dy).
\end{equation}
It follows that
\begin{equation}
\label{eq3.2}
\abs{(C_{\varphi}u)(x)} e^{-\Phi_2(x)} \leq \int_{\mathbb C^{n_1}} K(x,y) \abs{u(y)} e^{-\Phi_1(y)}\, L(dy),
\end{equation}
where
\begin{equation}
\label{eq3.3}
K(x,y) \leq C\, e^{2{\rm Re}\, \Psi_1(\varphi(x),\overline{y})} e^{-\Phi_2(x) - \Phi_1(y)}, \quad x\in\mathbb{C}^{n_2},y\in\mathbb{C}^{n_1}.
\end{equation}
In view of Schur's lemma, in order to show that $C_{\varphi} \in \mathcal L(H_{\Phi_1}(\mathbb C^{n_1}), H_{\Phi_2}(\mathbb C^{n_2}))$, it suffices to verify that
\begin{equation}
\sup_{x\in\mathbb{C}^{n_2}}\int_{\mathbb{C}^{n_1}} K(x,y)\, L(dy)<\infty, \quad \sup_{y\in\mathbb{C}^{n_1}}\int_{\mathbb{C}^{n_2}} K(x,y)\,L(dx)<\infty.
\label{integral bounds}
\end{equation}
To this end, using the fundamental estimate
\begin{equation}
\label{eq3.4}
2{\rm Re}\, \Psi_1(z,\overline{y}) - \Phi_1(z) - \Phi_1(y) = -\Phi_{\rm herm}(z-y) \asymp -\abs{z-y}^2, \quad y,z\in\mathbb{C}^{n_1},
\end{equation}
see~\cite[Chapter 13]{Zw12}, together with (\ref{eq3.3}) we obtain that
\begin{equation}
\label{eq3.5}
K(x,y) \leq C\, e^{-\abs{\varphi(x) -y}^2/C + \Phi_1(\varphi(x)) - \Phi_2(x)}, \quad C >0.
\end{equation}
It follows from (\ref{eq3.5}) that to establish \eqref{integral bounds}, we need to show that
\begin{equation}
\sup_{x\in\mathbb{C}^{n_2}}\int_{\mathbb{C}^{n_1}} e^{-\abs{\varphi(x)-y}^2/C +\Phi_1(\varphi(x))-\Phi_2(x)}\, L(dy)<\infty,
\label{y integral}
\end{equation}
\begin{equation}
\sup_{y\in\mathbb{C}^{n_1}}\int_{\mathbb{C}^{n_2}} e^{-\abs{\varphi(x)-y}^2/C +\Phi_1(\varphi(x))-\Phi_2(x)}\, L(dx)<\infty.
\label{x integral}
\end{equation}
Here (\ref{y integral}) is easily seen using (\ref{eq1.4}):
\begin{multline}
\label{eq3.6}
\int_{\mathbb{C}^{n_1}} e^{-\abs{\varphi(x)-y}^2/C +\Phi_1(\varphi(x))-\Phi_2(x)}\, L(dy) \leq C\,\int_{\mathbb{C}^{n_1}} e^{-\abs{\varphi(x)-y}^2/C}\, L(dy) \\
\leq C\, \int_{\mathbb{C}^{n_1}} e^{-\abs{y}^2/C}\, L(dy),
\end{multline}
and thus, (\ref{y integral}) holds. It only remains for us to prove \eqref{x integral}, which is the more interesting estimate, depending also on condition (\ref{kernel}) in Theorem \ref{theo 1}.

\medskip
\noindent
When showing (\ref{x integral}), we may ignore the complex structure, identifying $\mathbb{C}^{n_j} \simeq \mathbb{R}^{m_j}$, $m_j = 2n_j$, meaning that we now regard $\Phi_j$ as a real quadratic form on $\mathbb{R}^{m_j}$, $j=1,2$. We shall also view $\varphi$ as a real affine map from $\mathbb R^{m_2}\rightarrow\mathbb{R}^{m_1}$, regarding $A$ as a real $m_1\times m_2$ matrix and $b$ as an element of $\mathbb{R}^{m_1}$. We can introduce real linear coordinates on $\mathbb R^{m_2}$, still denoted by $x\in \mathbb R^{m_2}$, with the decomposition
\begin{equation}
\label{eq3.7}
x = (x_1,\ldots, x_{m_2}) = (x',x'')\in \mathbb R^{m_2-d} \times \mathbb R^{d},
\end{equation}
such that the null space ${\rm Ker}\, A \subset \mathbb R^{m_2}$ is given by the $d$ real linear equations
\begin{equation}
\label{eq3.8}
{\rm Ker}\, A=\{(x',x'')\in \mathbb R^{m_2};\,\, x''=0\}.
\end{equation}
Here $0 \leq d \leq m_2$. We need to show that
\begin{equation}
\label{eq3.9}
\sup_{y\in\mathbb{R}^{m_1}} \int_{\mathbb{R}^{d}} \int_{\mathbb{R}^{m_2-d}} e^{-\abs{\varphi(x',x'')-y}^2/C +\Phi_1(\varphi(x',x''))-\Phi_2(x',x'')}
\,dx'\, dx'' <\infty.
\end{equation}
Here we know in view of (\ref{eq3.8}) and (\ref{kernel condition}) that the quadratic form $\Phi_2(x',0)$ satisfies
\begin{equation}
\label{eq3.10}
\Phi_2(x',0) \asymp \abs{x'}^2, \quad x'\in \mathbb R^{m_2-d}.
\end{equation}

\medskip
\noindent
In order to proceed, we require the following essentially well known lemma, see~\cite[Proposition 2.4.3]{HiSj15}.
\begin{lemm}
\label{lemm_Taylor}
Consider a decomposition $\mathbb R^{m_2} \ni x = (x',x'') \in \mathbb R^{m_2-d} \times \mathbb R^d$, and let $q$ be a real quadratic form on $\mathbb R^{m_2}$ such that $x'\mapsto q(x',0)$ is a non-degenerate quadratic form on $\mathbb R^{m_2-d}$. Let $x'_c(x'') \in \mathbb R^{m_2-d}$ be the unique critical point of $\mathbb R^{m_2-d} \ni x' \mapsto q(x',x'')$, which is a linear function of $x'' \in \mathbb R^d$. We have
\begin{equation}
\label{eq3.11}
q(x',x'') = q(x'_c(x''),x'') +  q(x' - x'_c(x''),0), \quad (x',x'') \in \mathbb R^{m_2-d} \times \mathbb R^d.
\end{equation}
\end{lemm}
\begin{proof}
We have
\begin{equation}
\label{eq3.12}
\abs{q'_{x'}(x',0)} \asymp \abs{x'},\quad x'\in \mathbb R^{m_2-d},
\end{equation}
and hence the equation
\begin{equation}
\label{eq3.13}
q'_{x'}(x',x'') = 0
\end{equation}
has a unique solution $x' = x'_c(x'') \in \mathbb R^{m_2-d}$ for each $x'' \in \mathbb R^d$, depending linearly on $x''$. By Taylor's formula we can write
\begin{multline}
\label{eq3.14}
q(x',x'') = q(x'_c(x''),x'') + \frac{1}{2} q''_{x'x'}\,(x'- x'_c(x''))\cdot (x'-x'_c(x'')) \\
= q(x'_c(x''),x'') + q(x' - x'_c(x''),0).
\end{multline}
The proof is complete.
\end{proof}

\medskip
\noindent
We get, applying Lemma \ref{lemm_Taylor} to the quadratic form $\Phi_2$ and using (\ref{eq3.10}),
\begin{equation}
\label{eq3.15}
\Phi_2(x',x'') = \Phi_2(x'_c(x''),x'') + \Phi_2(x'-x'_c(x''),0) \geq \Phi_2(x'_c(x''),x'') + \frac{1}{C}\abs{x' - x'_c(x'')}^2,
\end{equation}
and therefore,
\begin{multline}
\label{eq3.16}
\Phi_1(\varphi(x',x''))-\Phi_2(x',x'') \leq \Phi_1(\varphi(x',x'')) - \Phi_2(x'_c(x''),x'') - \frac{1}{C}\abs{x' - x'_c(x'')}^2 \\
= \Phi_1(\varphi(x'_c(x''),x'')) - \Phi_2(x'_c(x''),x'') - \frac{1}{C}\abs{x' - x'_c(x'')}^2 \leq C - \frac{1}{C}\abs{x' - x'_c(x'')}^2.
\end{multline}
Here we have also used that
\begin{equation}
\label{eq3.17}
\varphi(x',x'') = A(x',x'') + b = A(0,x'') + b = \varphi(0,x'') = \varphi(x'_c(x''),x'')
\end{equation}
is independent of $x'$, in view of (\ref{eq3.8}), as well as the fact that $\Phi_1 \circ \varphi - \Phi_2$ is bounded above on $\mathbb R^{m_2}$.

\medskip
\noindent
It follows from (\ref{eq3.16}) and (\ref{eq3.17}) that the integral in (\ref{eq3.9}) does not exceed a constant times the integral
\begin{multline}
\label{eq3.18}
\int_{\mathbb{R}^{d}} \int_{\mathbb{R}^{m_2-d}} e^{-\abs{\varphi(0,x'')-y}^2/C - \abs{x' -x'_c(x'')}^2/C}\, dx'\, dx'' \\
= \left(\int_{\mathbb R^{m_2-d}} e^{-\abs{x'}^2/C}\, dx'\right) \left(\int_{\mathbb R^{d}} e^{-\abs{\varphi(0,x'')-y}^2/C}\, dx''\right).
\end{multline}
Here we have carried out the integration with respect to $x'$. It only remains for us to show therefore that
\begin{equation}
\label{eq3.19}
\sup_{y\in\mathbb{R}^{m_1}} \int_{\mathbb{R}^{d}} e^{-\abs{A(0,x'')-y}^2/C}\, dx'' <\infty,
\end{equation}
and to this end, we observe that
\begin{equation}
\label{eq3.20}
\abs{A(0,x'') - y}^2 \geq \abs{A(0,x'') - \Pi y}^2,
\end{equation}
where $\Pi: \mathbb{R}^{m_1}\rightarrow\mathbb{R}^{m_1}$ is the orthogonal projection onto ${\rm Im}\, A \subset \mathbb R^{m_1}$. We get
\begin{multline}
\label{eq3.21}
\sup_{y\in\mathbb{R}^{m_1}} \int_{\mathbb{R}^{d}} e^{-\abs{A(0,x'')-y}^2/C}\, dx'' \leq \sup_{z''\in\mathbb{R}^{d}} \int_{\mathbb{R}^{d}} e^{-\abs{A(0,x'')-A(0,z'')}^2/C}\, dx''\\
\leq \sup_{z''\in\mathbb{R}^{d}} \int_{\mathbb{R}^{d}} e^{-\abs{x''-z''}^2/\widetilde{C}}\, dx'' < \infty.
\end{multline}
Here we have also used that
\begin{equation}
\label{eq3.22}
\abs{A(0,u'')}^2 \asymp \abs{u''}^2, \quad u''\in \mathbb R^d,
\end{equation}
in view of (\ref{eq3.8}). The proof of Theorem \ref{theo 1} is complete.

\medskip
\noindent
{\it Remark}. Let us write following (\ref{eq3.1}), for $u\in H_{\Phi_1}(\mathbb C^{n_1})$,
\begin{equation}
\label{eq3.23}
C_{\varphi}u(x) = b_{\Phi_1}^2 \int\!\!\!\int_{\Gamma} e^{2(\Psi_1(\varphi(x),\theta) - \Psi_1(y,\theta))}\, u(y)\, \frac{dy\,d\theta}{(2i)^{n_1}}.
\end{equation}
Here $\Gamma \subset \mathbb C^{2n_1}_{y,\theta}$ is the contour of integration given by the anti-diagonal $\theta = \overline{y}$ and
\begin{equation}
\label{eq3.23.1}
\varphi(x) = Ax + b, \quad x\in \mathbb C^{n_2}.
\end{equation}
The holomorphic quadratic polynomial
\begin{equation}
\label{eq3.24}
F(x,y,\theta) = \frac{2}{i}\left(\Psi_1(\varphi(x),\theta) - \Psi_1(y,\theta)\right), \quad (x,y,\theta) \in \mathbb C^{n_2}_x \times \mathbb C^{n_1}_y \times \mathbb C^{n_1}_{\theta},
\end{equation}
is a non-degenerate phase function in the sense of H\"ormander~\cite{H71},
\begin{equation}
\label{eq3.25}
{\rm rank}\, \left(F''_{\theta x}\,\,\, F''_{\theta y}\,\,\, F''_{\theta \theta}\right) = n_1,
\end{equation}
since ${\rm det}\, (\Psi_{1})''_{\theta y} \neq 0$, and the operator $C_{\varphi}$ in (\ref{eq3.23}) can therefore be regarded as a metaplectic Fourier integral operator in the complex domain associated to the complex canonical relation
\begin{equation}
\label{eq3.26}
\kappa: \mathbb C^{2n_1}\ni (y, -F'_y(x,y,\theta)) \mapsto (x,F'_x(x,y,\theta))\in \mathbb C^{2n_2},\quad F'_{\theta}(x,y,\theta) = 0.
\end{equation}
Here we may recall from~\cite{CGHS} that the canonical relation (\ref{eq3.26}) is the graph of a (complex affine) canonical transformation precisely when $n_1=n_2=n$ and
\begin{equation}
\label{eq3.27}
{\rm det}\, \begin{pmatrix}
F''_{xy} & F''_{x\theta} \\
F''_{\theta y} & F''_{\theta \theta} \\
\end{pmatrix} \neq 0 \Longleftrightarrow {\rm det}\, F''_{x\theta} \neq 0,
\end{equation}
and in view of (\ref{eq3.23.1}), (\ref{eq3.24}), the latter condition is equivalent to the invertibility of $A$. Assuming that $A$ is invertible and introducing the real linear subspace
\begin{equation}
\label{eq3.28}
\Lambda_{\Phi_1} = \left\{\left(x,\frac{2}{i}\frac{\partial \Phi_1}{\partial x}(x)\right), \, x\in \mathbb C^n\right\} \subset \mathbb C^{2n} = \mathbb C^n_x \times \mathbb C^n_{\xi},
\end{equation}
we get, after a simple computation using (\ref{eq3.23.1}), (\ref{eq3.24}), and (\ref{eq3.26}),
\begin{equation}
\label{eq3.28.1}
\kappa: \mathbb C^{2n}\ni (y, \eta) \mapsto (A^{-1}(y-b), A^t \eta) \in \mathbb C^{2n},
\end{equation}
and
\begin{equation}
\label{eq3.29}
\kappa(\Lambda_{\Phi_1}) = \Lambda_{\Phi_1 \circ \varphi}.
\end{equation}
Here we may notice that the quadratic polynomial $\Phi_1 \circ \varphi$ is strictly plurisubharmonic on $\mathbb C^n$. The mapping property (\ref{eq3.29}) of the canonical transformation in (\ref{eq3.28.1}) corresponds to the operator mapping property
\begin{equation}
\label{eq3.30}
C_{\varphi}: H_{\Phi_1}(\mathbb C^n) \rightarrow H_{\Phi_1 \circ \varphi}(\mathbb C^n),
\end{equation}
immediate in the case when $A$ is bijective. In this work, representation (\ref{eq3.23}), allowing us to view $C_{\varphi}$ as a metaplectic Fourier integral operator in the complex domain, plays a crucial role in the proof of Theorem \ref{theo 1}. As we shall see in Section \ref{sect_compact}, arguments of the complex FIO theory continue to be essential when proving Theorem \ref{theo 2}.

\section{Characterizing compact pullback operators}
\label{sect_compact}
\medskip
\noindent
The purpose of this section is to prove Theorem~\ref{theo 2}. It is trivial that (\ref{trace}) implies (\ref{compact}) and we shall first check that (\ref{compact}) implies (\ref{neg def}). To this end, proceeding as in Section \ref{sec_nec} and continuing to follow~\cite{model}, we shall consider the action of the compact operator $C_{\varphi}^*: H_{\Phi_2}(\mathbb C^{n_2}) \rightarrow H_{\Phi_1}(\mathbb C^{n_1})$ on the space of coherent states $k_{2,w}$, $w\in \mathbb C^{n_2}$, given by (\ref{eq2.6}). We have $k_{2,w} \rightarrow 0$ weakly in $H_{\Phi_2}(\mathbb C^{n_2})$, as $\abs{w} \rightarrow \infty$, in view of~\cite[Lemma 3.1]{chs}, and therefore, $C_{\varphi}^*(k_{2,w}) \rightarrow 0$ in $H_{\Phi_1}(\mathbb C^{n_1})$, as $\abs{w} \rightarrow \infty$. Recalling (\ref{eq2.10}), we obtain that
\begin{equation}
\label{eq4.1}
e^{\Phi_1(\varphi(w)) - \Phi_2(w)} \rightarrow 0 \,\,\,\wrtext{as}\,\,\, \abs{w} \rightarrow \infty.
\end{equation}
An application of Lemma \ref{bounded implies affine} gives therefore that $\varphi(x) = Ax + b$, for some complex $n_1\times n_2$ matrix $A$ and some $b\in \mathbb C^{n_1}$, and then we have in view of (\ref{eq4.1}), (\ref{eq1.5}),
\begin{equation}
\label{eq4.2}
\Phi_1(Ax) - \Phi_2(x) + 2{\rm Re}\, \left((\partial_x \Phi_1)(Ax) \cdot b\right) \rightarrow -\infty,  \,\,\,\wrtext{as}\,\,\, \abs{x} \rightarrow \infty.
\end{equation}
It is clear that (\ref{eq4.2}) holds precisely when the quadratic form $\mathbb C^{n_2} \ni x \mapsto \Phi_1(Ax) - \Phi_2(x)$ is negative definite.

\medskip
\noindent
We shall finally prove that (\ref{trace}) is implied by (\ref{neg def}). When doing so we shall first show that the composition operator $C_{\varphi}$ is bounded,
\begin{equation}
\label{eq4.3}
C_{\varphi}: H_{\Phi_1}(\mathbb C^{n_1}) \rightarrow H_{\Phi_3}(\mathbb C^{n_2}),
\end{equation}
where $\Phi_3$ is a strictly plurisubharmonic quadratic form on $\mathbb C^{n_2}$ given by
\begin{equation}
\label{eq4.4}
\Phi_3(x) = \Phi_2(x) - \delta \abs{x}^2,
\end{equation}
for $\delta > 0$ small. Here the Bargmann space $H_{\Phi_3}(\mathbb C^{n_2})$ is defined as in (\ref{eq1.1.1}). To establish the mapping property (\ref{eq4.3}), we shall apply Theorem \ref{theo 1}. In light of assumption (\ref{eq1.8}), for sufficiently small $\delta > 0$, the quadratic form $\Phi_1(Ax) - \Phi_3(x)$ is negative definite,
\begin{equation}
\label{Phi3 neg def}
    \Phi_1(Ax) - \Phi_3(x) < 0,\quad 0\neq x\in\mathbb{C}^{n_2},
\end{equation}
so for any $0\neq x\in\mathbb{C}^{n_2}$ with $\Phi_3(x)\leq 0$, we must have $Ax \neq 0$. It only remains for us to check that
\begin{equation}
\label{eq4.5}
e^{\Phi_1 \circ \varphi - \Phi_3} \in L^{\infty}(\mathbb C^{n_2}),
\end{equation}
and to this end we write using (\ref{eq1.5}),
\begin{equation}
\label{eq4.6}
\Phi_1(\varphi(x)) - \Phi_3(x) = \Phi_1(Ax) - \Phi_3(x) + 2{\rm Re}\, \left((\partial_x \Phi_1)(Ax)\cdot b\right) + \Phi_1(b).
\end{equation}
Finally, (\ref{Phi3 neg def}) tells us precisely that the principal part of the quadratic polynomial in (\ref{eq4.6}) is negative definite on $\mathbb C^{n_2}$, and therefore (\ref{eq4.5}) holds. We have thus verified the mapping property (\ref{eq4.3}) for $\Phi_3$ in (\ref{eq4.4}) with $\delta > 0$ small enough.

\medskip
\noindent
The proof of Theorem \ref{theo 2} can now be completed by an application of the following general result, established in~\cite[Proposition 3.3]{CoHiSj19},~\cite[Corollary 2.6]{AlVi18}, using the theory of metaplectic Fourier integral operators in the complex domain,~\cite[Appendix B]{CGHS}.
\begin{prop}
\label{sv prop}
Let $\Phi_2$, $\Phi_3$ be strictly plurisubharmonic quadratic forms on $\mathbb C^{n_2}$ such that
\begin{equation}
\Phi_2(x) - \Phi_3(x) \asymp \abs{x}^2, \quad x\in \mathbb C^{n_2}.
\end{equation}
Then the inclusion map: $\iota: H_{\Phi_3}(\mathbb C^{n_2}) \rightarrow H_{\Phi_2}(\mathbb C^{n_2})$ is of trace class, with the singular values $s_j(\iota)$ satisfying
\[
s_j(\iota) \leq C \exp\left(-\frac{j^{1/{n_2}}}{C}\right), \quad j=1,2,\ldots,
\]
for some $C >0$.
\end{prop}

\end{document}